\documentclass[12pt,reqno]{amsart}

\usepackage{amssymb}

\newtheorem{theorem}{Theorem}[section]
\newtheorem{proposition}[theorem]{Proposition}
\newtheorem{lemma}[theorem]{Lemma}

\theoremstyle{remark}

\numberwithin{equation}{section}

\begin{document}

\title[Automorphisms of the generalized quot schemes]{Automorphisms of the generalized
quot schemes}

\author[I. Biswas]{Indranil Biswas}

\address{School of Mathematics, Tata Institute of Fundamental Research,
Homi Bhabha Road, Bombay 400005, India}

\email{indranil@math.tifr.res.in}

\author[S. Mehrotra]{Sukhendu Mehrotra}

\address{Facultad de Matem\'aticas, PUC Chile, Av. Vicu\~na Mackenna 4860, Santiago, Chile;
Chennai Mathematical Institute, H1, SIPCOT IT Park, Siruseri, Kelambakkam 603103, India}

\email{smehrotra@mat.uc.cl}

\subjclass[2010]{14H60, 14D21, 14D23}

\keywords{Generalized quot scheme, vector fields, automorphism group, symmetric product.}

\date{}

\begin{abstract}
Given a compact connected Riemann surface $X$ of genus $g\, \geq\, 2$, and integers 
$r\,\geq \, 2$, $d_p\, >\, 0$ and $d_z\, >\, 0$, in \cite{BDHW}, a generalized quot 
scheme ${\mathcal Q}_X(r,d_p,d_z)$ was introduced. Our aim here is to compute the 
holomorphic automorphism group of ${\mathcal Q}_X(r,d_p,d_z)$. It is shown that the 
connected component of $\text{Aut}( {\mathcal Q}_X(r,d_p,d_z))$ containing the 
identity automorphism is $\text{PGL}(r,{\mathbb C})$. As an application of it, we 
prove that if the generalized quot schemes of two Riemann surfaces are 
holomorphically isomorphic, then the two Riemann surfaces themselves are isomorphic.
\end{abstract}

\maketitle

\section{Introduction}\label{sec1}

In \cite{BDHW}, a generalized quot scheme was defined; we quickly recall it. Let $X$ 
be a compact connected Riemann surface of genus $g\, \geq\, 2$, and $r\,\geq \, 2$, 
$d_p\, >\, 0$ and $d_z\, >\, 0$ are integers. Let ${\mathcal Q}({\mathcal O}^{\oplus 
r}_X, d_p)$ be the quot scheme that parametrizes the coherent subsheaves of ${\mathcal 
O}^{\oplus r}_X$ of rank $r$ and degree $-d_p$. This complex projective variety
${\mathcal Q}({\mathcal O}^{\oplus r}_X, d_p)$ is a moduli space of
vortices \cite{BDW}, \cite{Br}, \cite{BR}, \cite{Ba}, \cite{EINOS}.
This moduli space is extensively studied (cf. \cite{Bifet}, \cite{BGL}).
The universal vector bundle over
$X\times {\mathcal Q}({\mathcal O}^{\oplus r}_X, d_p)$ will be denoted by
$\mathcal S$. The generalized quot scheme ${\mathcal Q}_X(r,d_p,d_z)$
parametrizes torsionfree coherent sheaves $F$ on $X$ of rank $r$ and degree $d_z-d_p$
such that some member of the family $\mathcal S$ is a subsheaf of $F$. In \cite{BDHW},
the fundamental group and the cohomology of ${\mathcal Q}_X(r,d_p,d_z)$ were computed.

The natural action of $\text{GL}(r,{\mathbb C})$ on the trivial vector bundle 
${\mathcal O}^{\oplus r}_X$ produces a holomorphic action of $\text{PGL}(r,{\mathbb 
C})$ on ${\mathcal Q}_X(r,d_p,d_z)$. The main result proved here says that
$\text{PGL}(r,{\mathbb C})$ is the connected component, containing the identity
element, of the group of holomorphic automorphisms of ${\mathcal Q}_X(r,d_p,d_z)$;
see Theorem \ref{thm1}.

Let $X'$ be a compact connected Riemann surface of genus at least two.
Fix positive integers $r'\,\geq \, 2$, $d'_p$ and $d'_z$. Let
${\mathcal Q}'_X(r',d'_p.d'_z)$
be the corresponding generalized quot scheme. As an application of Theorem
\ref{thm1}, we prove the following (see Proposition \ref{prop2}):

\begin{proposition}
If the two varieties ${\mathcal Q}'_X(r',d'_p.d'_z)$ and ${\mathcal Q}_X(r,d_p,d_z)$
are isomorphic, then $X$ is isomorphic to $X'$.
\end{proposition}

\section{Holomorphic vector fields on ${\mathcal Q}_X(r,d_p,d_z)$}

Let $X$ be a compact connected Riemann surface of genus $g$, with $g\, \geq\, 2$.
Fix positive integers $r\,\geq \, 2$, $d_p$ and $d_z$. Let ${\mathcal Q}({\mathcal O
}^{\oplus r}_X, d_p)$ denote the quot scheme that parametrizes all the torsion quotients
of ${\mathcal O}^{\oplus r}_X$
of degree $d_p$. Therefore, elements of ${\mathcal Q}({\mathcal O}^{\oplus r}_X, d_p)$
represent subsheaves $S$ of ${\mathcal O}^{\oplus r}_X$ such that ${\rm rank} (S)\,=\,r$
and ${\rm degree}(S)\,=\, -d_p$. These two conditions on the subsheaf $S$ are together
equivalent to the condition that ${\mathcal O}^{\oplus r}_X/S$ is torsion of degree
$d_p$. There is a universal short exact sequence of sheaves on $X\times {\mathcal
Q}({\mathcal O}^{\oplus r}_X, d_p)$
\begin{equation}\label{e1}
0\,\longrightarrow\, {\mathcal S} \,\longrightarrow\, p^*_X {\mathcal O}^{\oplus r}_X
\,=\, {\mathcal O}^{\oplus r}_{X\times {\mathcal Q}({\mathcal O}^{\oplus r}_X, d_p)}
\,\longrightarrow\, {\mathcal T}_1\,\longrightarrow\, 0\, ,
\end{equation}
where $p_X\, :\, X\times {\mathcal Q}({\mathcal O}^{\oplus r}_X, d_p)\,\longrightarrow\,
X$ is the natural projection.

Let
\begin{equation}\label{e2}
f\,:\, {\mathcal Q}\,=\, {\mathcal Q}_X(r,d_p,d_z)\,\longrightarrow\, 
{\mathcal Q}({\mathcal O}^{\oplus r}_X, d_p)
\end{equation}
be the relative quot scheme that parametrizes torsion quotients of ${\mathcal S}^*$
of degree $d_z$ (see \eqref{e1}). In other words, if $z$ is the point of ${\mathcal Q}({\mathcal O}^{\oplus
r}_X, d_p)$ representing a subsheaf $S\,\subset\, {\mathcal O}^{\oplus r}_X$, then the
fiber $f^{-1}(z)$ is the the space of all subsheaves of $S^*$ of rank $r$ and degree
$d_p-d_z$. Note that the degree of $S^*$ is $d_p$. Therefore, elements of ${\mathcal Q}$
parametrize diagrams of the form
\begin{equation}\label{e-1}
\begin{matrix}
&& 0 &&&&&&\\
&& \Big\downarrow &&&&&&\\
0& \longrightarrow & S & \longrightarrow & {\mathcal O}^{\oplus r}_X&
\longrightarrow & T_1 & \longrightarrow & 0\\
&& \Big\downarrow &&&&&&\\
&& V &&&&&&\\
&& \Big\downarrow &&&&&&\\
&& T_2 &&&&&&\\
&& \Big\downarrow &&&&&&\\
&& 0 &&&&&&
\end{matrix}
\end{equation}
where $T_1$ and $T_2$ are torsion sheaves of degree $d_p$ and $d_z$ respectively, and
the subsheaf $S\,\subset\, {\mathcal O}^{\oplus r}_X$ corresponds to the image, under
$f$, of the point of $\mathcal Q$ corresponding to $V$.

Let ${\rm Aut}({\mathcal Q})$ denote the group of all holomorphic automorphisms of
$\mathcal Q$. Since $\mathcal Q$ is a smooth complex projective variety, its holomorphic
automorphisms are automatically algebraic. Thus ${\rm Aut}({\mathcal Q})$ is a complex
Lie group with Lie algebra $H^0({\mathcal Q},\, T{\mathcal Q})$, where
$T{\mathcal Q}$ is the holomorphic tangent bundle of $\mathcal Q$; the Lie algebra
structure is given by the Lie bracket operation of vector fields. Let
$$
{\rm Aut}^0({\mathcal Q})\, \subset\, {\rm Aut}({\mathcal Q})
$$
be the connected component containing the identity element.

The standard action of $\text{GL}(r, {\mathbb C})$ on ${\mathbb C}^r$ produces
an action of $\text{GL}(r, {\mathbb C})$ on ${\mathcal O}^{\oplus r}_X$, because
the total space of ${\mathcal O}^{\oplus r}_X$ is identified with $X\times
{\mathbb C}^r$. This
action of $\text{GL}(r, {\mathbb C})$ on ${\mathcal O}^{\oplus r}_X$ defines an
action of $\text{GL}(r, {\mathbb C})$ on ${\mathcal Q}({\mathcal O}^{\oplus r}_X, d_p)$.
This action of $\text{GL}(r, {\mathbb C})$ on ${\mathcal Q}({\mathcal O}^{\oplus r}_X, d_p)$
evidently lifts to an action of $\text{GL}(r, {\mathbb C})$ on ${\mathcal Q}$ (see \eqref{e2}). Indeed, $\text{GL}(r, {\mathbb C})$ acts on diagrams of type \eqref{e-1}.
Since $\text{GL}(r, {\mathbb C})$ is connected, we get a homomorphism
$$
\text{GL}(r, {\mathbb C})\,\longrightarrow\, {\rm Aut}^0({\mathcal Q})\, .
$$
The center ${\mathbb C}^\star \,=\, {\mathbb C}\setminus \{0\}$ of $\text{GL}(r,{\mathbb
C})$ acts trivially on $\mathcal Q$. Hence the above homomorphism produces a homomorphism
\begin{equation}\label{e3}
\varphi\, :\, \text{PGL}(r, {\mathbb C})\,\longrightarrow\, {\rm Aut}^0({\mathcal Q})\, .
\end{equation}

\begin{theorem}\label{thm1}
The homomorphism $\varphi$ in \eqref{e3} is an isomorphism.
\end{theorem}

\begin{proof}
Let
\begin{equation}\label{e-2}
p\, :\, {\mathcal Q}\,\longrightarrow\, \text{Sym}^{d_p}(X)\times \text{Sym}^{d_z}(X)
\end{equation}
be the morphism that sends any $z\, \in\, \mathcal Q$ to the support of
$T_1$ with multiplicity given by $T_1$ and the support of
$T_2$ with multiplicity given by $T_2$, where $T_1$ and $T_2$ are the torsion sheaves
in the diagram \eqref{e-1} corresponding to the point $z$.

The homomorphism $\varphi$ is injective because the homomorphism
$$\text{PGL}(r, {\mathbb C})\,\longrightarrow\, {\rm Aut}({\mathbb C}{\mathbb P}^{r-1})$$
given by the standard action of $\text{GL}(r, {\mathbb C})$ on ${\mathbb C}^r$ is
injective. Indeed, for any
$$
x\,=\, ((x_1\, ,\cdots\, , x_{d_p})\, , (y_1\, ,\cdots\, , y_{d_z}))
\,\in\, \text{Sym}^{d_p}(X)\times \text{Sym}^{d_z}(X)\, ,
$$
where all the above $d_p+d_z$ points are distinct, the fiber $p^{-1}(x)$ is $({\mathbb C}
{\mathbb P}^{r-1})^{d_p}\times (CP^{r-1})^{d_z}$ (see \eqref{e-2}), and the action of
$\text{PGL}(r, {\mathbb C})$ on $p^{-1}(x)$ coincides with the diagonal action of
$\text{PGL}(r, {\mathbb C})$ on the factors in the above Cartesian product.

We need to prove that $\varphi$ is surjective.

The Lie algebra of $\text{PGL}(r, {\mathbb C})$ will be denoted by $\mathfrak g$;
it is the Lie algebra structure on trace zero $r\times r$--matrices with complex entries
given by commutator. Let
\begin{equation}\label{e4}
\theta\,:\, {\mathfrak g}\, \longrightarrow\, H^0({\mathcal Q},\, T{\mathcal Q})
\end{equation}
be the homomorphism of Lie algebras corresponding to the homomorphism $\varphi$ in
\eqref{e3}. To prove that $\varphi$ is surjective, it suffices to show that $\theta$
is surjective.

The following lemma is a key step in the computation of $H^0({\mathcal Q},\, T{\mathcal Q})$.

\begin{lemma}\label{lem1}
The Lie algebra $H^0({\mathcal Q},\, T{\mathcal Q})$ has a natural injective homomorphism
to ${\mathfrak g}\oplus\mathfrak g$.
\end{lemma}

\begin{proof}[Proof of Lemma \ref{lem1}]
For any positive integer $k$, let $P(k)$ denote the group of all permutations of
$\{1\, ,\cdots\, ,k\}$. Consider the action of $P(d_p)\times P(d_z)$ on
$X^{d_p+d_z}$ that permutes the first $d_p$ factors and the last $d_z$ factors. The
quotient $X^{d_p+d_z}/(P(d_p)\times P(d_z))$ is $\text{Sym}^{d_p}(X)\times
\text{Sym}^{d_z}(X)$. Let
$$
q\, :\, X^{d_p+d_z}\, \longrightarrow\, \text{Sym}^{d_p}(X)\times \text{Sym}^{d_z}(X)
$$
be the corresponding quotient map. 

Let
\begin{equation}\label{tU}
{\widetilde U}\, \subset\, X^{d_p+d_z}
\end{equation}
be the complement of the big diagonal,
so ${\widetilde U}$ parametrizes all possible distinct $d_p+d_z$ ordered points of
$X$. The image $q({\widetilde U})\, \subset\, \text{Sym}^{d_p}(X)\times \text{Sym}^{d_z}(X)$
will be denoted by $U$. Since $p^{-1}(U)$ is a
Zariski dense open subset of $\mathcal Q$, where $p$ is defined in \eqref{e-2}, we have
\begin{equation}\label{e5}
H^0({\mathcal Q},\, T{\mathcal Q})\,\subset\,H^0(p^{-1}(U),\, T(p^{-1}(U)))\, .
\end{equation}

The Galois group $\Gamma\,:=\, P(d_p)\times P(d_z)$ for the \'etale covering
$$q\vert_{\widetilde U}\, :\, \widetilde{U}\, \longrightarrow\, U\, ,$$
where $\widetilde U$ is defined in \eqref{tU}, acts on the fiber
product ${\mathcal Z}\, :=\, p^{-1}(U)\times_U \widetilde U$. We have
\begin{equation}\label{e6}
H^0(p^{-1}(U),\, T{\mathcal Q})\, =\, H^0({\mathcal Z},\, T{\mathcal Z})^\Gamma\, ,
\end{equation}
because the projection ${\mathcal Z}\, \longrightarrow\, p^{-1}(U)$ to the first
factor of the fiber product is an \'etale Galois covering with Galois group
$\Gamma$.

Now we have ${\mathcal Z}\,=\, \widetilde{U}\times ({\mathbb C}{\mathbb P}^{r-1})^{d_p}
\times ({\mathbb C}P^{r-1})^{d_z}$, where ${\mathbb C}P^{r-1}$ is the projective space
parametrizing the lines in ${\mathbb C}^{r}$. Note that
$$
H^0({\mathbb C}{\mathbb P}^{r-1},\, T{\mathbb C}{\mathbb P}^{r-1})\,=\,
{\mathfrak g}\,=\,H^0({\mathbb C}P^{r-1},\, T{\mathbb C}P^{r-1})\, .
$$
It is known that
$H^0(\widetilde{U},\, T\widetilde{U})\,=\, 0$ \cite[p. 1452, Proposition 2.3]{BDH}.
Also, we have $H^0(\widetilde{U},\, {\mathcal O}_{\widetilde U})\,=\, 0$
\cite[p. 1449, Lemma 2.2]{BDH}. These together imply that
$$
H^0({\mathcal Z},\, T{\mathcal Z})\,=\, {\mathfrak g}^{\oplus d_p}\oplus
{\mathfrak g}^{\oplus d_z}\, .
$$
The action of $\Gamma\,=\, P(d_p)\times P(d_z)$ on ${\mathfrak g}^{\oplus d_p}\oplus
{\mathfrak g}^{\oplus d_z}\,=\, H^0({\mathcal Z},\, T{\mathcal Z})$ (see \eqref{e6})
is the one that permutes first $d_p$ factors and the last $d_z$ factors.
Hence we have
$H^0({\mathcal Z},\, T{\mathcal Z})^\Gamma\, =\, {\mathfrak g}\oplus {\mathfrak g}$.
Therefore,
\begin{equation}\label{a}
H^0(p^{-1}(U),\, T(p^{-1}(U)))\,=\, {\mathfrak g}\oplus {\mathfrak g}
\end{equation}
by \eqref{e6}. Now the lemma follow from \eqref{e5}.
\end{proof}

Next we will need a property of the Hecke transformations.

Let $Y$ be a smooth complex algebraic curve and $y_0\,\in\, Y$ a point; the curve
$Y$ need not be projective.
Fix a linear nonzero proper subspace $0\,\not=\, S\, \subsetneq\, {\mathbb C}^r$. Consider
the short exact sequence of sheaves on $Y$
\begin{equation}\label{V}
0\,\longrightarrow\, V \,\longrightarrow\, {\mathcal O}^{\oplus r}_Y \,\longrightarrow\,
{\mathcal O}^{\oplus r}_{y_0}/S\,=\,  {\mathbb C}^r/S \,\longrightarrow\, 0\, .
\end{equation}
Let $P(V)$ denote the projective bundle over $Y$ that parametrizes the lines in the
fibers of $V$. Take any $A\, \in\, \text{GL}(r, {\mathbb C})$. Let ${\widehat A}$
be the automorphism of $P({\mathcal O}^{\oplus r}_Y)\,=\,Y\times CP^{r-1}$ given by
$A$; this automorphism acts trivially on $Y$ and has the standard action on
$CP^{r-1}$. Since $V$ and ${\mathcal O}^{\oplus r}_Y$ are identified over $Y\setminus
\{y_0\}$, the above automorphism $\widehat A$ produces an automorphism of
$P(V)\vert_{Y\setminus\{y_0\}}$. This automorphism of
$P(V)\vert_{Y\setminus\{y_0\}}$ will be denoted by ${\widehat A}'$.

\begin{lemma}\label{lem2}
The above automorphism ${\widehat A}'$ of $P(V)\vert_{Y\setminus\{y_0\}}$ extends
to a self-map of $P(V)$ if and only if $A(S)\,=\, S$.
\end{lemma}

\begin{proof}[Proof of Lemma \ref{lem2}]
Let $\text{GL}(V)$ be the Zariski open subset of the total space of
$\text{End}(V)\,=\, V\otimes V^*$ parametrizing endomorphisms of fibers
that are automorphisms. The quotient
$\text{PGL}(V)\,=\, \text{GL}(V)/{\mathbb C}^*$ is a group--scheme over $Y$
with fibers isomorphic to the group $\text{PGL}(r, {\mathbb C})$. If an algebraic
map of the total space
$$
\tau\, :\, P(V)\, \longrightarrow\, P(V)
$$
is an automorphism satisfies the condition that there is a nonempty Zariski open subset
$U_\tau\, \subset\, Y$ such that $\tau$ restricts to an automorphism of
$P(V)\vert_{U_\tau}$ over the identity map of $U_\tau$, then $\tau$ is actually an
automorphism over the identity map of $Y$. We note that the group of automorphisms of
$P(V)$ over the identity map of $Y$
is precisely the group of sections, over $Y$, of the group--scheme $\text{PGL}(V)$.

Fix a subspace $S'\,\subset\, {\mathbb C}^r$ complementary to $S$, so
${\mathbb C}^r\,=\, S\oplus S'$. Let $E_S\,:=\, Y\times S$ and
$E_{S'}\,:=\, Y\times S'$ be the trivial algebraic vector bundles over $Y$ with fibers
$S$ and $S'$ respectively. Then we have
\begin{equation}\label{e7}
{\mathcal O}^{\oplus r}_Y\,=\, E_S\oplus E_{S'}\ \ \text{ and }
V\,=\, E_S\oplus (E_{S'}\otimes {\mathcal O}_Y(-y_0))\, .
\end{equation}
{}From the above decompositions it follows immediately that if 
$A(S)\,=\, S$, then ${\widehat A}'$ extends to an automorphism of $P(V)$.

To prove the converse, assume that ${\widehat A}'$ extends to an
automorphism of $P(V)$. It suffices to show that the subbundle
$E_S\, \subset\, V$ in \eqref{e7} is preserved by the section of $\text{PGL}(V)$
corresponding to the automorphism of $P(V)$. Note that the restriction of this
section of $\text{PGL}(V)$ to $Y\setminus\{y_0\}$ is given by $A$.
There is no nonzero homomorphism from $E_S$ to $E_{S'}\otimes {\mathcal O}_Y(-y_0)$
which is given by a constant homomorphism
$B\, :\, S\, \longrightarrow\, S'$ on $Y\setminus\{y_0\}$ because such a homomorphism
over $Y\setminus\{y_0\}$ extends to a
homomorphism from $(E_S)_{y_0}$ to $(E_{S'})_{y_0}$ and this
homomorphism $(E_S)_{y_0}\,\longrightarrow\, (E_{S'})_{y_0}$ coincides with $B$.
On the other hand, the image of $(E_{S'}\otimes {\mathcal O}_Y(-y_0))_{y_0}$ in
$(E_{S'})_{y_0}$ is the zero subspace. So if
$B\,\not=\, 0$, then the homomorphism over $Y\setminus\{y_0\}$ does not extend
to a homomorphism from $E_S$ to $E_{S'}\otimes {\mathcal O}_Y(-y_0)$ over $Y$.
Therefore, we conclude that $A(S)\,=\, S$. 
\end{proof}

Fix distinct $d_p-1$ points $x_1\, ,\cdots\, , x_{d_p-1}$ on $X$. For each $x_i$,
fix a hyperplane $H_i$ in $({\mathcal O}^{\oplus r}_X)_{x_i}\,=\, {\mathbb C}^r$. Also,
fix distinct $d_z-1$ points $y_1\, ,\cdots\, , y_{d_z-1}$ on $X$ such that $x_i\,
\not=\, y_j$ for all $1\,\leq \, i\, \leq\, d_p-1$ and $1\,\leq \, j\, \leq\, d_z-1$.
Fix a line $L_j$ in $({\mathcal O}^{\oplus r}_X)_{y_j}\,=\, {\mathbb C}^r$ for each $j$.

Now take the complement $Y\,=\, X\setminus\{x_1\, ,\cdots\, , x_{d_p-1}\, , y_1\, 
,\cdots\, , y_{d_z-1}\}$. Take any nontrivial element
\begin{equation}\label{A}
{\rm Id}\,\not=\, A\,\in\, {\rm PGL}(r, {\mathbb C})\, .
\end{equation}
Fix a point $x_0\,\in\, Y$ and also fix a hyperplane 
$$S\,\subset\, ({\mathcal O}^{\oplus r}_X)_{x_0}\,=\, {\mathbb C}^r$$ such that
\begin{equation}\label{e8}
A(S)\,\not=\, S\, ;
\end{equation}
since $A\,\not=\, {\rm Id}$, such a subspace exists. Consider
the vector bundle $V$ on $Y$ constructed in \eqref{V} using $S$. As before, 
$P(V)$ denotes the projective bundle over $Y$ parametrizing the lines in the
fibers of $V$.

There is an embedding
\begin{equation}\label{de}
\delta\, :\, P(V)\, \longrightarrow\, {\mathcal Q}
\end{equation}
which we will now describe.
For the map $f$ in \eqref{e2}, the image $f\circ \delta (P(V))$ is the point
given by the quotient
$$
0\,\longrightarrow\, \widehat{V} \, \longrightarrow\, {\mathcal O}^{\oplus r}_X\, \longrightarrow\,
(\oplus_{i=1}^{d_p-1} ({\mathcal O}^{\oplus r}_X)_{x_i}/H_i)\oplus ({\mathcal O}^{\oplus
r}_X)_{x_0}/S)\, \longrightarrow\, 0\, ,
$$
where $H_i$ are the hyperplanes fixed above; in particular, $f\circ \delta$ is a constant
map. Note that $\widehat{V}$ is an extension of the vector bundle $V$ to $X$.
For any point $y\, \in\, Y$ and any point in the fiber $y'\, \in  P(V)_y$,
consider the short exact sequence on $X$
$$
0\,\longrightarrow\, E \, \longrightarrow\, \widehat{V}^*\, \longrightarrow\,
(\oplus_{j=1}^{d_z-1} (\widehat{V}^*_{y_j}/L^\perp_j)\oplus (\widehat{V}^*_y/L(y')^\perp)\,
\longrightarrow\, 0\, ,
$$
where $L(y')\, \subset\, V_y$ is the line in $V_y$ corresponding to the above point $y'$,
and $L^\perp_j\, \subset\, \widehat{V}^*_{y_j}$ (respectively,
$L(y')^\perp\,\subset\, \widehat{V}^*_y$) is the annihilator of $L_j$ (respectively,
$L(y')$); note that $\widehat{V}_{y_j}$ is identified with ${\mathbb C}^r$ and
$\widehat{V}_y$ is identified with $V_y$. Therefore, we have
$$
\widehat{V}\, \hookrightarrow\, E^*\, .
$$
The map $\delta$ in \eqref{de} sends any $y'$ to the above extension $E^*$ of $\widehat V$
constructed from $y'$.

{}From \eqref{a} we know that $\text{PGL}(r, {\mathbb C})$ is contained in 
$\text{Aut}(p^{-1}(U))$ with $$0\oplus {\mathfrak g}\, \subset\, {\mathfrak g}\oplus 
{\mathfrak g}\,=\, H^0(p^{-1}(U),\, T(p^{-1}(U)))$$ as its Lie algebra.
This action of $\text{PGL}(r, {\mathbb C})$ on $p^{-1}(U)$ clearly preserves the
intersection $\delta(P(V))\cap p^{-1}(U)$. Therefore, if the action of the element
$A$ in \eqref{A} extends to $\mathcal Q$, then the extended action
must preserve the image $\delta(P(V))$.

On the other hand, from Lemma \ref{lem2} we know that the action of $A$ on
$P(V)\vert_{Y\setminus \{y_0\}}$ does not extend to $P(V)$ because
\eqref{e8} holds. This completes the proof of the theorem.
\end{proof}

\section{Holomorphic maps from a symmetric product}

\begin{proposition}\label{prop1}
Let $X$ and $Y$ be compact connected Riemann surface with
$${\rm genus}(X)\, \geq\, {\rm genus}(Y)\, \geq\,2\, .$$ If there is a nonconstant
holomorphic map $\beta\, :\, {\rm Sym}^d(Y)\,\longrightarrow\, X$, then $d\,=\,1$, and
$\beta$ is an isomorphism.
\end{proposition}

\begin{proof}
Let $\beta\,:\, {\rm Sym}^d(Y)\,\longrightarrow\, X$ be a nonconstant holomorphic
map. Let
$$
\beta^*\, :\, H^0(X,\, \Omega^1_X)\,\longrightarrow\, H^0({\rm Sym}^d(Y),\,
\Omega^1_{{\rm Sym}^d(Y)})
$$
be the pull-back of $1$--forms defined by $\omega\,\longmapsto\, \beta^*\omega$.
This homomorphism $\beta^*$ is injective, because $\beta$ is surjective. Since
$$
\dim H^0({\rm Sym}^d(Y),\, \Omega^1_{{\rm Sym}^d(Y)})\,=\, \text{genus}(Y)
$$
\cite[p. 322, (4.3)]{Ma}, the injectivity of $\beta^*$ implies that ${\rm genus}(Y)\, \geq\,
{\rm genus}(X)$. Therefore, the given condition ${\rm genus}(X)\, \geq\, {\rm genus}(Y)$
implies that
\begin{itemize}
\item ${\rm genus}(X)\, =\, {\rm genus}(Y)$, and

\item the above homomorphism $\beta^*$ is an isomorphism.
\end{itemize}

If $d\, \geq\, 2$, the wedge product
$$
\wedge^2 H^0({\rm Sym}^d(Y),\, \Omega^1_{{\rm Sym}^d(Y)})\,\longrightarrow\,
H^0({\rm Sym}^d(Y),\, \Omega^2_{{\rm Sym}^d(Y)})
$$
is a nonzero homomorphism \cite[p. 325, (6.3)]{Ma}. On the other hand, the wedge product on
$H^0(X,\, \Omega^1_X)$ is the zero homomorphism because $H^0(X,\, \Omega^2_X)
\,=\, 0$. In other words, $\beta^*$ is not compatible with the wedge product operation
on holomorphic $1$-forms if $d\, \geq\, 2$. So we conclude that $d\,=\,1$.

Since ${\rm genus}(X)\, =\, {\rm genus}(Y)$, from Riemann--Hurwitz formula for
Euler characteristic if follows that $\text{degree}(\beta)\,=\,1$. In other words,
$\beta$ is an isomorphism.
\end{proof}

Let $X'$ be a compact connected Riemann surface of genus at least two.
Fix positive integers $r'\,\geq \, 2$, $d'_p$ and $d'_z$. Let
$$
{\mathcal Q}'\,=\, {\mathcal Q}'_X(r',d'_p.d'_z)
$$
be the corresponding generalized quot scheme (see \eqref{e2}).

\begin{proposition}\label{prop2}
If the two varieties ${\mathcal Q}'$ and ${\mathcal Q}$ (constructed in \eqref{e2})
are isomorphic, then $X$ is isomorphic to $X'$.
\end{proposition}

\begin{proof}
Assume that ${\mathcal Q}'$ and ${\mathcal Q}$ are isomorphic.
We will show that $X$ and $X'$ are isomorphic.

Let $\eta \,:\, \text{Sym}^{d_p}(X)\times \text{Sym}^{d_z}(X)\,\longrightarrow\,
\text{Pic}^{d_p}(X)\times \text{Pic}^{d_z}(X)$ be the morphism defined by
$$
((x_1\, ,\cdots\, , x_{d_p})\, , (y_1\, ,\cdots\, , y_{d_z}))\,\longmapsto\,
({\mathcal O}_X(x_1+\ldots + x_{d_p})\, ,{\mathcal O}_X(y_1+\ldots + y_{d_z}))\, .
$$
Since the general fiber of the map $p$ in \eqref{e-2} is a product of copies
of projective spaces, the composition
$$
\eta\circ p\, :\, {\mathcal Q}\, \longrightarrow\,
\text{Pic}^{d_p}(X)\times \text{Pic}^{d_z}(X)
$$
is the Albanese map for $\mathcal Q$, as there is no nonconstant holomorphic map from
a projective space to an abelian variety. In particular, the Albanese variety of $\mathcal
Q$ is of dimension $2g\,=\, 2\cdot \text{genus}(X)$. Therefore, comparing the Albanese
varieties of $\mathcal Q$ and ${\mathcal Q}'$ we conclude that $\text{genus}(X)\,=\,
g\,=\, \text{genus}(X')$.

Fix a maximal torus $T$ in $\text{Aut}^0({\mathcal Q})$. In view of Theorem \ref{thm1},
this amounts to choosing a trivialization of ${\mathcal O}_X^{\oplus r}$, with two
trivializations being identified if they differ by multiplication
with a constant nonzero scalar. The fixed--point
locus $${\mathcal Q}^T\, \subset\, \mathcal Q\, ,$$ for the action of $T$
on $\mathcal Q$, is a disjoint union of copies of
$$(\text{Sym}^{a_1}(X)\times \cdots \times \text{Sym}^{a_r}(X))\times
(\text{Sym}^{b_1}(X)\times \cdots \times\text{Sym}^{b_r}(X))$$
with $\sum_{i=1}^r a_i\,=\, d_p$
and $\sum_{i=1}^r b_i\,=\, d_z$.

Take a component $$Z \, =\, (\text{Sym}^{a_1}(X)\times \cdots\times\text{Sym}^{a_r}(X))
\times (\text{Sym}^{b_1}(X)\times \cdots\times\text{Sym}^{b_r}(X))$$ of ${\mathcal Q}^T$
such that at least one of the $2r$ integers $\{a_1\, ,\cdots\, ,a_r\, ,b_1\, ,\cdots\, ,
b_r\}$ is one.

We will first show that $Z$ is not holomorphically isomorphic to $\text{Sym}^{c_1}(Y)\times
\cdots\times\text{Sym}^{c_n}(Y)$, where $Y$ is compact connected Riemann surface of genus $g$
and $c_j\, \geq\, 2$ for all $1\,\leq\, j\,\leq\, n$. To prove this, assume that
$Z$ is isomorphic to $\text{Sym}^{c_1}(Y)\times
\cdots\times\text{Sym}^{c_n}(Y)$, where $Y$ and $c_i$ are as above. Consider the
composition
$$
\text{Sym}^{c_1}(Y)\times\cdots\times\text{Sym}^{c_n}(Y)\, \stackrel{\sim}{\longrightarrow}\,
Z\, \stackrel{q}{\longrightarrow}\, X\, ,
$$
where $q$ is the projection to a factor of $Z$ which is the first symmetric power of $X$
(it is assumed that such a factor exists). Since all $c_j\, \geq \, 2$, from Proposition
\ref{prop1} it follows that there is no nonconstant map from $\text{Sym}^{c_1}(Y)\times
\cdots\times\text{Sym}^{c_n}(Y)$ to $X$. Therefore, we conclude that $Z$ is not
holomorphically isomorphic to $\text{Sym}^{c_1}(Y)\times\cdots\times\text{Sym}^{c_n}(Y)$.

Fix a maximal torus $T'\, \subset\, \text{Aut}({\mathcal Q}')^0$. Since ${\mathcal Q}'$ is
isomorphic to $\mathcal Q$, there is a component
$$
(\text{Sym}^{a'_1}(X')\times \cdots \times \text{Sym}^{a'_r}(X'))\times
(\text{Sym}^{b'_1}(X')\times \cdots \times\text{Sym}^{b'_r}(X'))
$$
of the fixed point locus $({\mathcal Q}')^{T'}\, \subset\, {\mathcal Q}'$ which is
isomorphic to $Z$. Now from Proposition \ref{prop1} it follows that
\begin{itemize}
\item at least one of the $2r$ integers $\{a'_1\, ,\cdots\, ,a'_r\, ,b'_1\, ,\cdots\, ,
b'_r\}$ is one, and

\item $X$ is isomorphic to $X'$.
\end{itemize}
This completes the proof.
\end{proof}

\section*{Acknowledgements}

The first named author thanks Facultad de Matem\'aticas, PUC Chile, for hospitality.
He also acknowledges the support of a J. C. Bose Fellowship.  
The second named author was partially supported by FONDECYT grant 1150404 
during the preparation of this article.


\end{document}